\documentclass{article}
\usepackage{amsfonts}
\usepackage{amsmath}
\usepackage{amsthm}
\usepackage{amssymb}
\usepackage{mathrsfs}
\usepackage{graphicx}

\def\ep{\varepsilon}

\newtheorem{theorem}{Theorem}[section]
\newtheorem{definition}[theorem]{Definition}
\newtheorem{lemma}[theorem]{Lemma}
\newtheorem{corollary}[theorem]{Corollary}

\begin{document}

\title{\LARGE{On metric characterizations of some classes of Banach spaces}}

\author{Mikhail I.~Ostrovskii}

\maketitle

\noindent{\bf Abstract.} The paper contains the following results
and observations: (1) There exists a sequence of unweighted graphs
$\{G_n\}_n$ with maximum degree $3$ such that a Banach space $X$
has no nontrivial cotype iff $\{G_n\}_n$ admit uniformly
bilipschitz embeddings into $X$; (2) The same for Banach spaces
with no nontrivial type; (3) A sequence $\{G_n\}$ characterizing
Banach spaces with no nontrivial cotype in the sense described
above can be chosen to be a sequence of bounded degree expanders;
(4) The infinite diamond does not admit a bilipschitz embedding
into Banach spaces with the Radon-Nikod\'{y}m property; (5) A new
proof of the Cheeger-Kleiner result: The Laakso space does not
admit a bilipschitz embedding into Banach spaces with the
Radon-Nikod\'{y}m property; (6) A new proof of the
Johnson-Schechtman result: uniform bilipschitz embeddability of
finite diamonds into a Banach space implies its
nonsuperreflexivity.
\medskip

\medskip

\medskip

\begin{large}

\section{Introduction}

A mapping $F:X\to Y$ between metric spaces $X$ and $Y$ is called a
{\it $C$-bilipschitz embedding} if there exists $r>0$ such that
$\forall u,v\in X\quad rd_X(u,v)\le d_Y(F(u),F(v))\le rCd_X(u,v)$.
A sequence $\{f_n\}$ of mappings $f_n:U_n\to Z_n$ between metric
spaces is a sequence of {\it uniformly bilipschitz embeddings} if
there is $C<\infty$ such that all of the embeddings are
$C$-bilipschitz.

Many important classes of Banach spaces have been characterized in
terms of uniformly bilipschitz embeddings of finite metric spaces.
Bourgain \cite{Bou86} proved that a Banach space $X$ is
nonsuperreflexive if and only if there exist uniformly bilipschitz
embeddings of finite binary trees $\{T_n\}_{n=1}^\infty$ of all
depths into $X$. Similar characterization of spaces with no type
$>1$ was obtained by Bourgain, Milman, and Wolfson \cite{BMW86}: a
Banach space $X$ has no type $>1$ if and only if there exist
uniformly bilipschitz embeddings of Hamming cubes
$\{H_n\}_{n=1}^\infty$ into $X$ (see \cite{Pis86} for a simpler
proof). Johnson and Schechtman \cite{JS09} found a similar
characterization of nonsuperreflexive spaces in terms of diamond
graphs \cite{GNRS04} and Laakso graphs \cite{Laa00}. Banach spaces
without cotype were characterized by Mendel and Naor \cite{MN08}
in terms of lattice graphs $L_{m,n}$ whose vertex sets are
$\{0,1,\dots,m\}^n$, two vertices are joined by an edge if and
only if their $\ell_\infty$-distance is equal to $1$.\medskip

Characterizations in terms of bilipschitz embeddability of certain
metric spaces are called {\it metric characterizations}. Observe
that binary trees and Laakso graphs are graphs with uniformly
bounded degrees. Degrees of Hamming cubes and the lattice graphs
are unbounded. During the seminar ``Nonlinear geometry of Banach
spaces'' (Workshop in Analysis and Probability at Texas A~\&~M
University, 2009) Johnson posed the following problem: Find metric
characterizations of spaces with no type $p>1$ and with no cotype
in terms of graphs with uniformly bounded degrees. The first part
of this paper is devoted to a solution of this problem.\medskip

In the second part of the paper we prove results related to
another problem posed by Johnson during the mentioned seminar:
Find metric characterizations of reflexivity and the
Radon-Nikod\'{y}m property (RNP). We prove that Banach spaces
containing bilipschitz images of the infinite diamond do not have
the RNP, but the converse is not true. We find a new proof of the
Cheeger-Kleiner \cite{CK09} result that Banach spaces containing
bilipschitz images of the Laakso space do not have the RNP.
\medskip

The author would like to thank Florent Baudier, William
B.~Johnson, Mikhail M.~Po\-pov, Beata Randrianantoanina, and
Gideon Schechtman for their interest to this work and for useful
related discussions.

\section{Type and cotype in terms of graphs with uniformly bounded
degrees}

\begin{theorem}\label{T:deg3} There exist metric characterizations of the
classes of spaces with no type $>1$ and with no cotype in terms of
graphs with maximum degree $3$.
\end{theorem}

We start with the characterization of cotype. This case is easier,
because the well-known results (\cite[Proposition 15.6.1]{Mat02}
and \cite{MP76}) imply that any family of finite metric spaces is
uniformly bilipschitz embeddable into any Banach space with no
cotype. Therefore to prove the cotype part of the theorem we need
to show only that the graphs $L_{m,n}$ are uniformly bilipschitz
embeddable into a family of graphs with uniformly bounded degrees.
Thus it suffices to prove the following lemma.

\begin{lemma}\label{L:GM} Let $(V(G),d_G)$ be the vertex set of a graph $G$ with its
graph distance, and let $\ep>0$. Then there exist a graph $M$ with
maximum degree $\le 3$, $\ell\in\mathbb{N}$, and an embedding
$F:V(G)\to V(M)$ such that
\[\ell d_G(u,v)\le d_M(F(u),F(v))\le (1+\ep)\ell d_G(u,v)\quad\forall u,v\in V(G),\]
where $d_M$ is the graph distance of $M$.
\end{lemma}

\begin{proof} Let $\Delta_G$ be the maximum degree of $G$
and let $r\in\mathbb{N}$ be such that $3\cdot 2^{r-1}\ge\Delta_G$.
Let $\ell\in\mathbb{N}$ be such that
$\frac{\ell+2r}{\ell}<1+\varepsilon$. We define the graph $M$ in
the following way. For each vertex $v$ of $G$ the graph $M$
contains a $3$-regular tree of depth $r$ rooted at a vertex which
we denote $m(v)$. For each edge $uv$ of $G$ we pick a leaf of the
tree rooted at $m(v)$ and a leaf of the tree rooted at $m(u)$ and
join them by a path of length $\ell$. Leaves picked for different
edges are different (this is possible because $3\cdot
2^{r-1}\ge\Delta_G$), and there is no further interaction between
the constructed trees and paths. It is easy to see that the
maximum degree of $M$ is $3$.
\medskip

We map $V(G)$ into $V(M)$ by mapping each $v$ to the corresponding
$m(v)$. It remains to show that \begin{equation}\label{E:GvsM}\ell
d_G(u,v)\le d_M(m(u),m(v))\le (\ell+2r)d_G(u,v)\end{equation}

The right-hand side inequality follows from the observation that
if $u$ and $v$ are adjacent in $G$, then
$d_M(m(u),m(v))\le\ell+2r$, the path of length $\ell+2r$ can be
constructed as the union of path in $M$ corresponding to $uv$, and
paths from $m(u)$ and $m(v)$ to the corresponding leaves.
\medskip

To prove the left-hand side of \eqref{E:GvsM} we consider a path
joining $m(u)$ and $m(v)$. Let $m(u_1),\dots,m(u_k)$ be the set of
roots of those trees which are visited by the path, listed in the
order of visits. The description of $M$ implies that
$u,u_1,\dots,u_k,v$ is a $uv$-walk in $G$, hence its length is
$\ge d_G(u,v)$. In order to get from one tree to another we need
to traverse $\ell$ edges. Hence any path joining $m(u)$ and $m(v)$
has length $\ge \ell d_G(u,v)$. This completes the proof of the
lemma and the cotype part of the theorem.\end{proof}

\begin{proof}[Proof of the type part of Theorem \ref{T:deg3}] By \cite[Theorem
2.3]{MP76}, each space with no type $>1$ contains subspaces whose
Banach-Mazur distances to $\ell_1^d$ $(d\in\mathbb{N})$ are
arbitrarily close to $1$. Therefore it suffices to check that each
of the graphs obtained in a similar way from Hamming cubes admits
a uniformly bilipschitz embedding into $\ell_1^d$ for sufficiently
large $d$. We denote by $\{S_n\}_{n=1}^\infty$ graphs obtained
from $\{H_n\}_{n=1}^\infty$ using the procedure described in the
proof of Lemma \ref{L:GM}. We describe an embedding of the vertex
set of $S_n$ into $\ell_1^k$. The images of vertices of $S_n$
under this embedding are integer points of $\ell_1^k$, edges of
$S_n$ correspond to line segments of length $1$ parallel to unit
vectors of $\ell_1^k$. Having such a representation of $S_n$, it
remains to show that the identity mappings of vertex sets of $S_n$
endowed with their graph distances and their $\ell_1$-distances
are uniformly bilipschitz.\medskip

The graph $H_n$ is $n$-regular, so we let $r\in\mathbb{N}$ be such
that $n\le 3\cdot 2^{r-1}$ and consider a rooted $3$-regular tree
of depth $r$. This tree can be isometrically embedded into
$\ell_1^m$, where $m=3+3\cdot 2+\dots+3\cdot 2^{r-1}$. The
embedding is the following: observing that $m$ is the number of
edges in the tree, we find a bijection between unit vectors in
$\ell_1^m$ and edges of the tree. Now we map the root of the tree
to $0\in\ell_1^m$; if $v$ is different from the root, we map $v$
to the sum of unit vectors corresponding to the path from $v$ to
the root. We denote by $T_r$ the image of the tree in
$\ell_1^m$.\medskip

We consider the natural isometric embedding of $H_n$ into
$\ell_1^n$, with images of the vertices being all possible
$0,1$-sequences. We pick $\ell$ in the same way as in Lemma
\ref{L:GM}. We specify the position of the rooted tree
corresponding to the vertex $v=\{\theta_i\}_{i=1}^n$ of $H_n$ in
$\ell_1^k=\ell_1^m\oplus_1\ell_1^n$ as
$T_r+\ell\cdot\{\theta_i\}$, where we mean that
$T_r\subset\ell_1^m$ and $\ell\cdot\{\theta_i\}$ is a multiple of
$v$ considered as a vector in $\ell_1^n$.\medskip

We introduce and embed the paths of length $\ell$ (from the
construction of Lemma \ref{L:GM}) in the following way: Since $n$
is $\le$ the number of leaves in $T_r$, there is a bijection
between the unit vectors of $\ell_1^n$ and some subset of leaves
of $T_r$. On the other hand, each edge of $H_n$ is parallel to one
of the unit vectors. We add $\ell$-paths in the following way. The
path corresponding to the edge between $v$ and $v+e_t$ ($e_t$ is a
unit vector of $\ell_1^n$) is the straight line path of length
$\ell$ joining the leaves of $T_r+\ell v$ and $T_r+\ell(v+e_t)$;
in each of the trees the leaf is chosen in such a way that its
$\ell_1^m$ component is the leaf corresponding to $e_t$.\medskip

It is clear that the graph $S_n$ obtained in this way fits the
description of $M$ in the proof of Lemma \ref{L:GM}. Therefore the
natural embedding of $H_n$ into $S_n$ is $(1+\ep)$-bilipschitz if
both graphs are endowed with their graph distances. It remains to
estimate the bilipschitz constants of natural embeddings of $S_n$
into $\ell_1^k$.\medskip

Observe that the graph distance between two vertices of $S_n$
cannot be less than the distance between their images in
$\ell_1^n\oplus_1\ell_1^{m}$, because each edge corresponds to a
line segment of length $1$. It remains to show that the graph
distance between two vertices of $S_n$ cannot be much larger than
the $\ell_1$-distance. Let $x,y$ be two vertices of $S_n$, we need
to estimate $d_{S_n}(x,y)$ from above in terms of $||x-y||_1$.
\medskip

For each set of vertices of the form $T_r+\ell v$ we consider its
union with the set of all vertices of $\ell$-paths going out of
this set.  It is easy to see that if both $x$ and $y$ belong to
one of such sets, then $d_{S_n}(x,y)\le ||x-y||_1$.
\medskip

For $x\in V(S_n)$ denote the projection of $x$ to $\ell_1^n$ by
$\pi(x)$, and the $i$-th coordinate of this projection by
$\pi(x)_i$. If the situation described in the previous paragraph
does not occur then there exists $i\in\{1,\dots,n\}$ such that
$|\pi(x)_i-\pi(y)_i|=\ell$. Let $k\le n$ be the number of
coordinates for which this equality holds.\medskip

We have $||x-y||_1\ge ||\pi(x)-\pi(y)||_1\ge k\ell$. To estimate
$d_{S_n}(x,y)$ from above we construct the following $xy$-path in
$S_n$. If one of the numbers $\pi(x)_i$ is strictly between $0$
and $\ell$ we start by moving from $x$ in the direction of
$\pi(y)_i$ (which in this case should be $0$ or $\ell$) till we
reach a set of the form $T_r+\ell w_x$ for some vertex $w_x$ of
$H_n$.
\medskip

We do similar thing at the other end of the path (near $y$): If
one of the numbers $\pi(y)_i$ is strictly between $0$ and $\ell$
we end the path by moving from $y$ in the direction of $\pi(x)_i$
(which in this case should be $0$ or $\ell$) till we reach a set
of the form $T_r+\ell w_y$.
\medskip

We find a shortest path between $w_x$ and $w_y$ in $H_n$. It is
easy to see that it has length $k$. Now we continue construction
of the $xy$-path in $S_n$. This path will contain all paths of
length $\ell$ corresponding to the edges of the $w_xw_y$-path in
$H_n$. Between these paths we add the pieces of the corresponding
trees of the from $T_r+\ell u$, needed to make a path. As a result
we get an $xy$-path of length $<2\ell+k\ell+2(k+1)r$. If $4r\le
\ell$ (we can definitely assume this), we have
$2\ell+k\ell+2(k+1)r\le 4k\ell$. In such a case $d_{S_n}(x,y)\le
4||x-y||_1$. This completes the proof of the type part of the
theorem.\end{proof}

\begin{corollary} There exists a family $\{K_n\}$ of constant degree
expanders, such that a Banach space $X$ for which there exist
uniformly bilipschitz embeddings of $\{K_n\}$ into $X$, has no
cotype.
\end{corollary}

\begin{proof} Let $\{M_n\}$ be graphs of maximum degree $3$
from the metric characterization of Banach spaces with no cotype.
It suffices to show that there exists a family $\{K_n\}$ of
constant degree expanders containing subsets isometric to
$\{M_n\}$. Consider  any family $\{G_k\}$ of constant $d$-regular
expanders with the growing number of vertices. Let $D_n$ be the
diameter of $M_n$ and $m_n$ be its number of vertices. It is clear
that we may assume without loss of generality that $G_n$ contains
a $D_n$-separated set of cardinality $m_n$. We fix a bijection
between this set and $V(M_n)$. We add to $G_n$ edges between
vertices corresponding to adjacent vertices of $M_n$. Since $D_n$
is the diameter of $M_n$, the obtained graph contains an isometric
copy of $M_n$. The maximum degree of the obtained graph is $\le
d+3$. Its expanding properties are not worse than those of $G_n$.
Adding as many self-loops to it as is needed we get a
$(d+3)$-regular graph $K_n$. It is clear that $\{K_n\}$ is a
desired family of $(d+3)$-regular expanders.
\end{proof}

\section{Diamonds and Laakso graphs}

The {\it diamond graph} of level $0$ is denoted $D_0$. It has two
vertices joined by an edge of length $1$. $D_i$ is obtained from
$D_{i-1}$ as follows. Given an edge $uv\in E(D_{i-1})$, it is
replaced by a quadrilateral $u, a, v, b$ with edge lengths
$2^{-i}$. We endow $D_n$ with their shortest path metrics. We
consider the vertex of $D_n$ as a subset of the vertex set of
$D_{n+1}$, it is easy to check that this defines an isometric
embedding. We introduce $D_\omega$ as the union of the vertex sets
of $\{D_n\}_{n=0}^\infty$. For $u,v\in D_\omega$ we introduce
$d_{D_\omega}(u,v)$ as $d_{D_n}(u,v)$ where $n\in\mathbb{N}$ is
any integer for which $u,v\in V(D_n)$. Since the natural
embeddings $D_n\to D_{n+1}$ are isometric, it is easy to see that
$d_{D_n}(u,v)$ does not depend on the choice of $n$ for which
$u,v\in V(D_n)$.\medskip

\begin{definition}[\cite{Jam72} or {\cite[p.~34]{Bou83}}]\label{D:tree}  {\rm Let $\delta>0$.
A sequence $\{x_i\}_{i=1}^\infty$ is called a {\it $\delta$-tree}
if  $x_i=\frac12(x_{2i}+x_{2i+1})$ and
$||x_{2i}-x_i||=||x_{2i+1}-x_i||\ge\delta$.}\end{definition}

\begin{theorem}\label{T:DiamTree} If $D_\omega$ is bilipschitz embeddable into a
Banach space $X$, then $X$ contains a bounded $\delta$-tree for
some $\delta>0$.
\end{theorem}

It is well-known that Banach spaces with the RNP do not contain
bounded $\delta$-trees (see \cite[p.~31]{Bou83}). On the other
hand there exist Banach spaces without the RNP which do not
contain bounded $\delta$-trees, see \cite[p.~54]{BR80}. So Theorem
\ref{T:DiamTree} implies:

\begin{corollary} If $D_\omega$ is bilipschitz embeddable into a
Banach space $X$, then $X$ does not have the Radon-Nikod\'{y}m
property. The converse is not true.
\end{corollary}

\begin{proof}[Proof of Theorem \ref{T:DiamTree}]
Let $f:D_\omega\to X$ be a bilipschitz embedding. Without loss of
generality we assume that
\begin{equation}\label{E:delta}
\delta d_{D_\omega}(x,y)\le||f(x)-f(y)||\le d_{D_\omega}(x,y)
\end{equation}
for some $\delta>0$.
\medskip

Let us show that this implies that the unit ball of $X$ contains a
$\delta$-tree. The first element of the tree will be
$x_1=f(u_0)-f(v_0)$, where $\{u_0,v_0\}=V(D_0)$.
\medskip

Now we consider the quadrilateral $u_0,a,v_0,b$. Inequality
\eqref{E:delta} implies $||f(a)-f(b)||\ge\delta$. Consider two
pairs of vectors (corresponding to two different paths from $u$ to
$v$ in $D_1$):
\medskip

\noindent{\bf Pair 1:} $f(v_0)-f(a)$, $f(a)-f(u_0)$.\qquad {\bf
Pair 2:} $f(v_0)-f(b)$, $f(b)-f(u_0)$.
\medskip

The inequality $||f(a)-f(b)||\ge\delta$ implies that at least one
of the following is true
\[||(f(v_0)-f(a))-(f(a)-f(u_0))||\ge\delta\quad \hbox{ or }\quad
||(f(v_0)-f(b))-(f(b)-f(u_0))||\ge\delta.
\]

Suppose that the first inequality holds. We let
\[x_2=2(f(v_0)-f(a))\quad \hbox{and}\quad x_3=2(f(a)-f(u_0)).\] It is clear that both
conditions of Definition \ref{D:tree} are satisfied. Also, the
condition \eqref{E:delta} implies that $||x_2||,||x_3||\le 1$.
\medskip

We continue construction of the $\delta$-tree in the unit ball of
$X$ in a similar manner. For example, to construct $x_4$ and $x_5$
we consider the corresponding quadrilateral $a,a_1,v_0,b_1$ in
$D_2$. The inequality $||f(a_1)-f(b_1)||\ge\delta/2$ implies that
at least one of the following is true
\[||(f(v_0)-f(a_1))-(f(a_1)-f(a))||\ge\delta/2\hbox{ or }
||(f(v_0)-f(b_1))-(f(b_1)-f(a))||\ge\delta/2.
\]

Suppose that the second inequality holds. We let
\[x_4=4(f(v_0)-f(b_1))\quad \hbox{and}\quad x_5=4(f(b_1)-f(a)).\] It is clear that both
conditions of Definition \ref{D:tree} are satisfied. Also
\eqref{E:delta} implies that $||x_4||,||x_5||\le 1$. Proceeding in
an obvious way we get a $\delta$-tree in the unit ball of $X$.
\end{proof}

\subsection{Finite version and the Johnson-Schechtman
characterization of superreflexivity}

\begin{definition}[\cite{Jam72}] {\rm A Banach space $X$
has the {\it finite tree property} if there exist $\delta>0$ such
that for each $k\in\mathbb{N}$ the unit ball of $X$ contains a
finite sequence $\{x_i:~i=1,\dots,2^{k}-1\}$ such that
$x_i=\frac12(x_{2i}+x_{2i+1})$ and
$||x_{2i}-x_i||=||x_{2i+1}-x_i||\ge\delta$ for each
$i=1,\dots,2^{k-1}-1$.}
\end{definition}

It is clear that the proof of Theorem \ref{T:DiamTree} implies its
finite version:

\begin{corollary}\label{C:FinDiamTree} If there exist uniformly bilipschitz embeddings of
$\{D_n\}_{n=1}^\infty$ into a Banach space $X$, then $X$ has the
finite tree property.
\end{corollary}

Combining Corollary \ref{C:FinDiamTree} with the well-known fact
(see \cite{Jam72} and \cite{Enf72}) that the finite tree property
is equivalent to nonsuperreflexivity, we get the second part of
the result in \cite[p.~181]{JS09}: uniform bilipschitz
embeddability of $\{D_n\}_{n=1}^\infty$ into $X$ implies the
nonsuperreflexivity of $X$.

\subsection{Laakso space}

Our version of the Laakso space (originally constructed in
\cite{Laa00}) is similar to the version from \cite[p.~290]{LP01}.
However, our version is a countable set (dense in the version of
the space from \cite{LP01}). The {\it Laakso graph} of level $0$
is denoted $L_0$. It consists of two vertices joined by an edge of
length $1$. The {\it Laakso graph} $L_i$ is obtained from
$L_{i-1}$ as follows. Each edge $uv\in E(L_{i-1})$ of length
$4^{-i+1}$ is replaced by a graph with $6$ vertices $u, t_1, t_2,
o_1, o_2, v$ where $o_1, t_1, o_2, t_2$ form a quadrilateral, and
there are only two more edges $ut_1$ and $vt_2$, with all edge
lengths $4^{-i}$. %We call the pair $o_1o_2$ an {\it anti-edge} corresponding to $uv$.
We endow $L_n$ with their shortest path metrics. We consider the
vertex of $L_n$ as a subset of the vertex set of $L_{n+1}$, it is
easy to check that this defines an isometric embedding. We
introduce the {\it Laakso space} $L_\omega$ as the union of the
vertex sets of $\{L_n\}_{n=0}^\infty$. For $u,v\in L_\omega$ we
introduce $d_{L_\omega}(u,v)$ as $d_{L_n}(u,v)$ where
$n\in\mathbb{N}$ is any integer for which $u,v\in V(L_n)$. Since
the natural embeddings $L_n\to L_{n+1}$ are isometric, it is easy
to see that $d_{L_n}(u,v)$ does not depend on the choice of $n$
for which $u,v\in V(L_n)$.\medskip

Our next purpose is to give a new proof of the following result of
Cheeger and Kleiner \cite[Corollary 1.7]{CK09}:

\begin{theorem}\label{T:LaakTree} If $L_\omega$ is bilipschitz embeddable into a
Banach space $X$, then $X$ does not have the Radon-Nikod\'{y}m
property.
\end{theorem}

\begin{proof} We do not know whether bilipschitz embeddability of
$L_\omega$ into $X$ implies the existence of a bounded
$\delta$-tree in $X$. To prove Theorem \ref{T:LaakTree} we
introduce the following definition.

\begin{definition}\label{D:semitree}  {\rm Let $\delta>0$.
A sequence $\{x_i\}_{i=1}^\infty$ is called a {\it
$\delta$-semitree} if
$x_i=\frac14(x_{4i-2}+x_{4i-1}+x_{4i}+x_{4i+1})$ and
$||(x_{4i-2}+x_{4i-1})-(x_{4i}+x_{4i+1})||\ge\delta$.}\end{definition}

Our proof has two steps. First we show that bilipschitz
embeddability of $L_\omega$ into $X$ implies that $X$ contains a
bounded $\delta$-semitree. The second step is to show that
existence of a bounded $\delta$-semitree in $X$ implies that $X$
does not have the RNP (this is almost standard, based on
martingales).
\medskip

Let $f:L_\omega\to X$ be a bilipschitz embedding. Without loss of
generality we assume that
\begin{equation}\label{E:delta2}
\delta d_{L_\omega}(x,y)\le||f(x)-f(y)||\le d_{L_\omega}(x,y)
\end{equation}
for some $\delta>0$.
\medskip

We need to construct a $\delta$-semitree in the unit ball of $X$.
The first element of the semitree is $x_1=f(u_0)-f(v_0)$, where
$\{u_0,v_0\}=V(L_0)$.
\medskip

Now we consider the $4$-tuple $u_0,o_1,v_0,o_2$. Observe that
\eqref{E:delta2} together with $d_{L_\omega}(o_1,o_2)\ge 1/2$
implies that $||o_1-o_2||\ge\delta/2$. Consider two pairs of
vectors:
\medskip

\noindent{\bf Pair 1:} $f(v_0)-f(o_1)$, $f(o_1)-f(u_0)$. \qquad
{\bf Pair 2:} $f(v_0)-f(o_2)$, $f(o_2)-f(u_0)$.
\medskip

The inequality $||f(o_1)-f(o_2)||\ge\delta/2$ implies that at
least one of the following is true
\[||(f(v_0)-f(o_1))-(f(o_1)-f(u_0))||\ge\delta/2 \hbox{ or }
||(f(v_0)-f(o_2))-(f(o_2)-f(u_0))||\ge\delta/2.
\]

Suppose that the first inequality holds. We let
\[x_2=4(f(v_0)-f(t_2)), x_3=4(f(t_2)-f(o_1)), x_4=4(f(o_1)-f(t_1)), x_5=
4(f(t_1)-f(u_0)).\] It is easy to check that both conditions of
Definition \ref{D:semitree} are satisfied, we even get
\[||(x_{2}+x_{3})-(x_{4}+x_{5})||=4||(f(v_0)-f(o_1))-(f(o_1)-f(u_0))||\ge2\delta.\] Also,
\eqref{E:delta2} applied to
$d_{L_\omega}(u_0,t_1)=d_{L_\omega}(t_1,o_1)=d_{L_\omega}(o_1,t_2)=d_{L_\omega}(t_2,v_0)=1/4$
implies that $||x_2||,||x_3||,||x_4||,||x_5||\le 1$.
\medskip

We continue our construction of the $\delta$-semitree in the unit
ball of $X$ in a similar manner. For example, to construct $x_6$,
$x_7$, $x_8$, and $x_9$, we consider the $6$-tuple corresponding
to the edge $t_2v_0$ of $L_1$ and repeat the same procedure as
above for $u_0v_0$. Proceeding in an obvious way we get a
$\delta$-semitree in the unit ball of $X$.
\medskip

To show that presence of a bounded $\delta$-semitree implies
absence of the RNP we use the same argument as for $\ep$-bushes in
\cite[p.~111]{BL00}. We construct an $X$-valued martingale
$\{f_n\}_{n=0}^\infty$ on $[0,1]$. We let $f_0=x_1$. The function
$f_2$ is defined on four quarters of $[0,1]$ by $x_2,x_3,x_4,x_5$,
respectively. To define the function $f_3$ we divide $[0,1]$ into
$16$ equal subintervals, and define $f_3$ as $x_6,\dots,x_{21}$,
on the respective subintervals, etc.

It is clear that we get a sequence of uniformly bounded functions.
The first condition in the definition of a $\delta$-semitree
implies that this sequence is a martingale. The second condition
implies that it is not convergent almost everywhere because it
shows that on each interval of the form
$\left[\frac{k}{4^n},\frac{k+1}{4^n}\right]$ the average value of
$||f_n-f_{n+1}||$ over the first half of the interval is
$\ge\delta/4$, this implies that
$||f_n(t)-f_{n+1}(t)||\ge\delta/4$ on a subset in $[0,1]$ of
measure $\ge\frac12$. It remains to apply \cite[Theorem
5.8]{BL00}.
\end{proof}

\begin{small}

\medskip

\noindent{\sc Department of Mathematics and Computer Science\\
St. John's University\\ 8000 Utopia Parkway, Queens, NY 11439,
USA}\\
e-mail: {\tt ostrovsm@stjohns.edu}

\end{small}

\end{large}
\end{document}